\documentclass[reqno,11pt,final]{amsart}

\usepackage[latin1]{inputenc}
\usepackage{amsmath,amssymb,amsthm}
\usepackage{bbm,subfigure}
\usepackage{enumitem,epsfig,psfrag}
\setenumerate{leftmargin=*}
\setitemize{leftmargin=*}

\usepackage[colorlinks,bookmarks,linkcolor=black,citecolor=black]{hyperref}
\usepackage[notref,notcite,color]{showkeys}

\let\subset\subseteq 
\let\eps\varepsilon
\let\epsilon\varepsilon
\let\rho\varrho

\def\bigO{\mathcal{O}}

\def\clqed{\hfill\scalebox{.6}{$\Box$}}

\def\itmit#1{\rm({\it #1\,})}
\def\rom{\itmit{\roman{*}}}
\def\abc{\itmit{\alph{*}}}

\newtheorem{theorem}{Theorem}
\newtheorem{lemma}[theorem] {Lemma}

\theoremstyle{remark}

\newcommand{\oldqed}{}

\newcommand{\EMAIL}[1]{  \textit{E-mail}: \texttt{#1}} 

\newcommand{\By}[2]{\overset{\mbox{\tiny{#1}}}{#2}}
\newcommand{\ByRef}[2]{   \By{\eqref{#1}}{#2} }

\newcommand{\leByRef}[1]{ \ByRef{#1}{\le} }
\newcommand{\geByRef}[1]{ \ByRef{#1}{\ge} }

\newcommand{\Tur}[2]{\mathrm{T}_{#1}({#2})} 
\newcommand{\tur}[2]{t_{#1}({#2})}          
\newcommand{\turm}[3]{t_{#1}({#2,#3})}
\newcommand{\Turm}[3]{\mathcal{T}_{#1}({#2,#3})}

\title{An extension of Tur\'an's Theorem, uniqueness and stability}
  \author[P. Allen]{Peter Allen*}
  \author[J. B\"ottcher]{Julia B\"ottcher*}

  \thanks{
    *
    Department of Mathematics,
    London School of Economics,
    Houghton Street,
    London, WC2A 2AE, UK.
    \EMAIL{p.d.allen|j.boettcher@lse.ac.uk}}

  \author[J. Hladk\'y]{Jan Hladk\'y\dag}

  \thanks{
    \dag\ 
    DIMAP and Mathematics Institute,
    University of Warwick,
    Coventry, CV4~7AL, UK. The author is an EPSRC Research Fellow.
    \EMAIL{honzahladky@gmail.com}}

  \author[D. Piguet]{Diana Piguet\ddag}

\thanks{\ddag\ 
New Technologies for Information Society,
University of West Bohemia,
Pilsen, Czech Republic.
\EMAIL{piguet@ntis.zcu.cz}}
  \thanks{
    PA, JH, and DP were supported by DIMAP, EPSRC award
    EP/D063191/1.
    PA was partially supported by FAPESP (Proc.~2010/09555-7), and
    JB by FAPESP (Proc.~2009/17831-7).
    PA and JB are grateful to NUMEC/USP, N\'ucleo de Modelagem Estoc\'astica e
    Complexidade of the University of S\~ao Paulo, for supporting this
    research. The research leading to this result has received funding from the
    European Union's Seventh Framework Programme (FP7/2007-2013) under grant
    agreement no.\ PIEF-GA-2009-253925.}

    \date{\today}

\date{\today}

\begin{document}
\begin{abstract}
  We determine the maximum number of edges of an $n$-vertex graph $G$ with
  the property that none of its $r$-cliques intersects a fixed set
  $M\subset V(G)$.  For $(r-1)|M|\ge n$, the $(r-1)$-partite Tur\'an graph
  turns out to be the unique extremal graph. For $(r-1)|M|<n$, there is a
  whole family of extremal graphs, which we describe explicitly. In
  addition we provide corresponding stability results.
\end{abstract}

\maketitle

\section{Introduction}

Tur\'an's Theorem~\cite{Tur}, whose proof marks the beginning of Extremal
Graph Theory, determines the maximum number of edges of $n$-vertex graph
without a copy of the $r$-clique $K_r$. It turns out that the \emph{unique}
extremal graph for this problem is the \emph{Tur\'an graph} $\Tur{r-1}{n}$,
that is, the complete balanced $(r-1)$-partite graph on $n$-vertices. We
write $\tur{r-1}{n}$ to denote the number of edges of~$\Tur{r-1}{n}$.

Tur\'an's Theorem is a primal example of a \emph{stable} result: The
Erd\H{o}s-Simonovits stability theorem~\cite{E68,S68} asserts that any
$n$-vertex $K_r$-free graph with almost $\tur{r-1}{n}$ edges looks very
similar to $\Tur{r-1}{n}$. In order to make this more precise we need the
following definition.  We say that an $n$-vertex graph $G$ is
\emph{$\eps$-close} to a graph $H$ on the same vertex set if the \emph{edit distance}\footnote{Edit distance is a well-studied concept in graph theory, see for example the introduction of~\cite{Martin}.} between $G$ and $H$ is at most $\eps n^2$, that is, if $H$ can be
obtained from $G$ by editing (deleting/inserting) at most $\eps n^2$ edges
and relabelling the vertices. In this case we also say that $G$ is
\emph{$(\eps n^2)$-near} to $H$.

\begin{theorem}[Erd\H{o}s~\cite{E68} \&
Simonovits~\cite{S68}]\label{thm:BasicStability} Suppose that $r\geq 3$ and $\eps^*>0$ are given. Then there exists
$\gamma^*>0$ such that each $\ell$-vertex graph $G$ with no $K_r$ and
$e(G)>\tur{r-1}{\ell}-\gamma^* \ell^2$ is $\epsilon^*$-close to
$\Tur{r-1}{\ell}$.
\end{theorem}

In fact, Erd\H{o}s and Simonovits both proved more general statements, allowing any fixed $r$-partite graph $H$ in place of $K_r$. Moreover, in more recent years strengthenings have been proved, for example that most vertices of any~$G$ as in Theorem~\ref{thm:BasicStability} are in an induced $(r-1)$-partite graph,~\cite{NikiPlacate1}. There are also further generalisations, such as obtaining the same conclusion as in Theorem~\ref{thm:BasicStability} while allowing the size of the forbidden subgraph $H$ to depend on $v(G)$,~\cite{NikiPlacate2}.

A main motivation for proving stability results for extremal statements is that they are often useful in applications where the original extremal statement would not suffice. This is for example the case when the Szemer\'edi Regularity Lemma (see, e.g., the survey~\cite{KS96}) is used. A prominent example of such an application is the enumeration result of Balogh, Bollob\'as and Simonovits~\cite{BalBolSim} giving a precise count of $H$-free graphs. It is worth observing that in most applications the `basic' stability theorem of Erd\H{o}s and Simonovits, Theorem~\ref{thm:BasicStability}, suffices.

Our goal is to extend Tur\'an's Theorem, by determining the maximum
number of edges in an $n$-vertex graph $G$ such that no copy of $K_r$
in~$G$ touches a fixed vertex set $M\subset V(G)$ of size
$m$. It turns out that for $(r-1)m\ge n$ the unique extremal graph is
$\Tur{r-1}{n}$. 
The case $(r-1)m< n$ is more complicated. In particular, there is a whole
family of extremal graphs, which we describe in Section~\ref{sec:Extremal} below. In both cases we shall
denote the (family of) extremal graphs by $\Turm{r-1}{n}{m}$, and their
number of edges by
\begin{equation}\label{eq:turm}
  \turm{r-1}{n}{m}:=\begin{cases}
    \tur{r-1}{n}\;,
    & \text{if $n\le(r-1)m$}\,, \\
    \binom{n}{2} - nm + (r-1)\binom{m+1}{2}\;,
    & \text{otherwise}\,.
  \end{cases}
\end{equation}

Our two main results are as follows.

\begin{theorem}\label{thm:TuranExt} 
  Given $r\geq 3$ and $m\leq n$, let $G$ be any $n$-vertex graph and
  $M\subset V(G)$ contain $m$ vertices, such that no copy of $K_r$ in $G$
  intersects $M$. Then 
  \begin{enumerate}[label=\abc]
    \item\label{thm:TuranExt:bound} $e(G)\le\turm{r-1}{n}{m}$, and 
    \item\label{thm:TuranExt:unique} if $e(G)=\turm{r-1}{n}{m}$ then $G\in \Turm{r-1}{n}{m}$.
  \end{enumerate}
\end{theorem}

Theorem~\ref{thm:TuranExt}\ref{thm:TuranExt:unique} states that the graphs
$\Turm{r-1}{n}{m}$ we construct below are the only extremal graphs.
The following theorem provides a corresponding stability result.

\begin{theorem}\label{thm:TuranExt_stab}
  Given $r\geq 3$ and $\eps>0$ there exists $\gamma>0$ such that the following holds. Let $m\leq n$, let $G$ be any $n$-vertex graph and $M\subset V(G)$ contain $m$ vertices, such that no copy of $K_r$ in $G$ intersects $M$.  If $e(G)> \turm{r-1}{n}{m}-\gamma n^2$, then $G$ is $\eps$-close to a graph from $\Turm{r-1}{n}{m}$ in which no copies of~$K_r$ intersect~$M$.
\end{theorem}

We remark that
Theorem~\ref{thm:TuranExt}\ref{thm:TuranExt:bound} is also included in our
previous paper~\cite{AllBoettHlaPig_Turannical}, but we did not determine the
family of extremal graphs there.\footnote{Actually, at the time of writing~\cite{AllBoettHlaPig_Turannical} we believed that the family of extremal graphs described there was complete. Only later we discovered further constructions involving `sporadic vertices' (see below).} Hence our main contribution here is to determine the extremal graphs and prove stability. This, however, turns out to be an important tool
for~\cite{AllBottHlaPig}, where we determine the maximum number of edges in an
$n$-vertex graph without a given number of vertex-disjoint triangles. Note that
the statement of Theorem~\ref{thm:TuranExt_stab} gives a
slightly stronger version of stability than the usual one, namely that the set
$M$ is not changed in transforming $G$ to a member of $\Turm{r-1}{n}{m}$.
We require this in~\cite{AllBottHlaPig}.

We note that the proof of Theorem~\ref{thm:TuranExt}\ref{thm:TuranExt:bound}
as given in~\cite{AllBoettHlaPig_Turannical} hints the main arguments involved
in our proof of Theorem~\ref{thm:TuranExt}. However, several additional
tweaks and tricks are needed, in particular in the
case $n>(r-1)m$. We give an outline of the proofs of Theorem~\ref{thm:TuranExt} and~\ref{thm:TuranExt_stab} in Section~\ref{ssec:outline}.

\medskip
The $(r-1)m\ge n$ case of Theorem~\ref{thm:TuranExt} shows that the assumption
in Tur\'an's Theorem (or in that of Theorem~\ref{thm:BasicStability}) can be
substantially weakened from forbidding $K_r$-copies on all possible $r$-subsets
of the vertex set~$V(G)$, to just forbidding $K_r$-copies on a particular family
$\mathcal S$ of $r$-subsets---the family~$\mathcal S$ which contains all
$r$-subsets of~$V(G)$ which intersect~$M$.
In~\cite{AllBoettHlaPig_Turannical} we investigated such weakenings of the
assumption in Tur\'an's theorem also from a probabilistic perspective. In
particular, we proved that forbidding $K_r$-copies on a \emph{random} family of
$r$-sets $\mathcal S\subset\binom{n}{r}$ of size only $|\mathcal S|=\bigO(n^3)$
suffices.

\subsection{Extremal graphs}\label{sec:Extremal}

The family \index{$\Turm{r-1}{n}{m}$}$\Turm{r-1}{n}{m}$ is defined as
follows. As previously stated, if $n\le (r-1)m$ then $\Turm{r-1}{n}{m}=\big\{\Tur{r-1}{n}\big\}$. 
So assume from now on that $n>(r-1)m$. We explicitly describe the construction of
the graphs in~$\Turm{r-1}{n}{m}$.

We start with the Tur\'an graph $\Tur{r-1}{(r-1)m}$, with colour classes $V_1$,
\dots, $V_{r-1}$, and an arbitrary set $M$ of $m$ vertices in $V_1\cup\cdots\cup
V_{r-1}$. We add $r-1$ new vertices $v_1,\ldots,v_{r-1}$ to this graph with the
following property. For each $i\in[r-1]$, the vertex $v_i$ is adjacent to all
old and new vertices except those in $V_i$ (and itself). Finally, we add a set
$Y$ of $n-(r-1)m$ new vertices each of which is adjacent to all old and new
vertices except those in $M$ (and itself). In this way we obtain an
$(n+r-1)$-vertex graph, which we call $G_r(n,M)$. Note that the graph $G_r(n,M)$
depends on the placement of $M$ in $V_1\cup\cdots \cup V_{r-1}$. We let
$\Turm{r-1}{n}{m}$ be the family of $n$-vertex graphs which can be obtained from
 some graph $G_r(n,M)$ by deleting any $r-1$ vertices from
$\{v_1,\ldots,v_{r-1}\}\cup Y$ (see also Figure~\ref{fig:Turm}). We call the
vertices $v_1,\ldots,v_{r-1}$ \emph{sporadic}.

\begin{figure}[t]
     \centering
     \subfigure{
          \includegraphics[width=.28\textwidth]{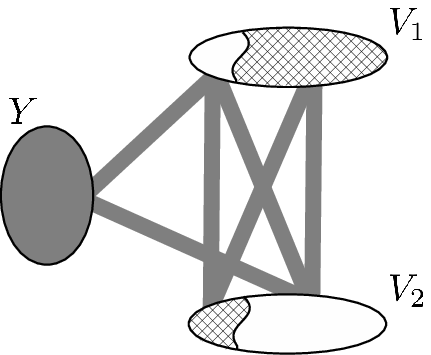}
     }
     \hspace{.12in}
     \subfigure{
          \includegraphics[width=.28\textwidth]{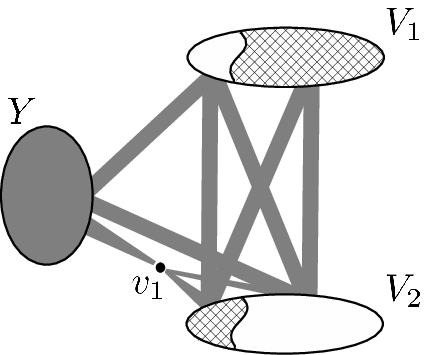}
     }
     \hspace{.12in}
     \subfigure{
         \includegraphics[width=.28\textwidth]{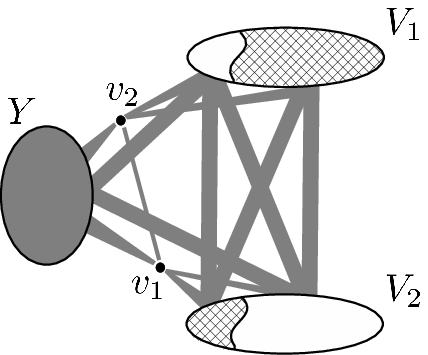}
     }
     \label{fig:Turm}
     \caption{Examples of graphs from~$\Turm{2}{n}{m}$ for $n>2m$, with no, one, and two sporadic vertices. Grey depicts complete (bipartite) graphs, the set $M$ is hatched.}
\end{figure}

Observe that there is no copy of $K_r$ in $G_r(n,M)$ which uses vertices of $M$.
Furthermore, the vertices $\{v_1,\ldots,v_{r-1}\}\cup Y$ form a clique in
$G_r(n,M)$, and each of these vertices has degree $n+r-2-m$. It follows that indeed every
graph in $\Turm{r-1}{n}{m}$ has the same number of edges, and that number is
\[\binom{n-(r-1)m}{2}+\big(n-(r-1)m\big)(r-2)m+\tur{r-1}{(r-1)m}=\turm{r-1}{n}{m}\,,\]
as desired.

\section{Proofs of Theorem~\ref{thm:TuranExt} and Theorem~\ref{thm:TuranExt_stab}}

\subsection{Outline of the proofs.}\label{ssec:outline}
We prove Theorem~\ref{thm:TuranExt} and Theorem~\ref{thm:TuranExt_stab} together.
We refer to the cases $n\le(r-1)m$ and  $n>(r-1)m$ as Cases~I and~II,
respectively. We prove Case~I first, and then prove Case~II using Case~I.
 
In Case~I, we sequentially pick maximum vertex disjoint cliques $P_1,\ldots,P_k$ of order at least~$r$. Because of their sizes, we know they do not intersect the set~$M$. A counting argument gives an upper bound on the number of edges in~$G$, depending on the sizes of these cliques (see Lemma~\ref{lem:largem}). This upper bound is enough to prove Theorem~\ref{thm:TuranExt} in Case~I. Further, we infer from Lemma~\ref{lem:largem} that if $e(G)>\tur{r-1}{n}-o(n^2)$ then the total order of the cliques $P_1,\ldots,P_k$ must be $o(n)$. Therefore, $e(G-\bigcup_i P_i)>e(G)-o(n^2)\ge\tur{r-1}{n}-o(n^2)$, and the Erd\H{o}s--Simonovits Stability Theorem, Theorem~\ref{thm:BasicStability}, applies to the graph $G-\bigcup_i P_i$. Thus, the graph $\Tur{r-1}{n-o(n)}$ is similar to $G-\bigcup_i P_i$, which in turn is similar to $G$, as needed.

Let us note that
even though Theorem~\ref{thm:TuranExt} extends Tur\'an's theorem, the counting argument in Lemma~\ref{lem:largem} actually relies on Tur\'an's result.
 
The proof strategy for Theorem~\ref{thm:TuranExt} in Case~II comes naturally from the structure of the extremal graphs. The key property to observe is that in these graphs, the neighborhood of the set $M$ induces essentially $\Tur{r-1}{(r-1)|M|}$ (with the exception of the sporadic vertices). As a first step, we apply a Zykov-type symmetrisation to our graph $G$ with no copy of $K_r$ intersecting $M$ (Lemma~\ref{lem:push}) to obtain a graph $G'$. We then perform a further simple transformation to remove any sporadic vertices, obtaining a graph $G''$ with at least as many edges as $G$. Now we can show that the union of $M$ and its neighbourhood in $G''$ cover at most $(r-1)|M|$ vertices (Lemma~\ref{lem:boundX}). This means that we can apply the bound from Case~I on the union of $M$ and its neighbourhood, and trivial bounds on edges in other parts of the graph, to conclude part~\ref{thm:TuranExt:bound}. To prove~\ref{thm:TuranExt:unique} we observe that equality is possible only if the union of $M$ and its neighbourhood in $G''$ is $\Tur{r-1}{(r-1)|M|}$ by Case~I of part~\ref{thm:TuranExt:unique} and the trivial bounds are sharp, in which case $G''$ is in $\Turm{r-1}{n}{m}$. This implies that $G'$ is also in $\Turm{r-1}{n}{m}$ (Lemma~\ref{lem:Seq}), and finally we conclude that $e(G)=e(G')$ only if $G=G'$ (Lemma~\ref{lem:push}), as required.

The proof of Case~II of Theorem~\ref{thm:TuranExt_stab} follows the same pattern as the uniqueness result, using Case~I of Theorem~\ref{thm:TuranExt_stab} to show that $G''$ is close in edit distance to a graph in $\Turm{r-1}{n}{m}$ and then we show that the same is true of $G$ (Lemma~\ref{lem:push}).
 
\subsection{Case~I}

The following lemma will be the key tool for proving uniqueness and
stability when $n\le(r-1)m$.

  \begin{lemma}\label{lem:largem}
    Given $m$ and $n\le(r-1)m$, let $G$ be an $n$-vertex graph and $M$ a
    subset of $V(G)$ with $|M|=m$ such that no copy of $K_r$ in $G$ uses
    vertices of $M$. Suppose that there are sets $P_1,\ldots,P_k$ of sizes
    $p_1,\ldots,p_k$ in $G$ such that the following holds for all $i\in[k]$.
    \begin{enumerate}[label=\rom]
      \item $|P_i|\ge r$.
      \item $P_i$ is the vertex set of a maximum clique in
        $G\big[V(G)\setminus\cup_{j=1}^{i-1}P_j\big]$.
     \item $G\big[V(G)\setminus\cup_{j=1}^{k}P_j\big]$ contains no
        $K_r$.
    \end{enumerate}
    Let $p:=\sum_{\ell=1}^k p_\ell$.
    Then we have
    \begin{equation*}
      e(G)\le \tur{r-1}{n}-\sum_{i=1}^k
      \sum_{j=0}^{p_i-r} \Big(m-\Big\lfloor\frac{n-\big(p-j-\sum_{\ell=i+1}^{k}p_\ell\big)}{r-1}\Big\rfloor-1\Big)\,.
    \end{equation*}
  \end{lemma}
  \begin{proof}
    We first
    establish some simple bounds on the number of edges in $G$.
    Each $P_i$ contains $\binom{p_i}{2}$ edges. By the maximality of
    $P_1,\ldots,P_k$ we have $\deg(v,P_i)\le p_i-1$ for any
    $v\in V(G)\setminus \bigcup_{j=1}^{i}V(P_i)$.
    Because no copy of $K_r$ in $G$ intersects $M$, we have $M\subset
    V(G)\setminus \bigcup_{i=1}^k V(P_i)$ and the stronger bound $\deg(v,P_i)\le
    r-2$ for each $v\in M$. Finally, since the graph
    $G-\cup_{i=1}^k P_i$ is $K_r$-free, by Tur\'an's theorem we have
    \begin{equation}\label{eq:BasicTuran}
      e\Big(G-\bigcup_{i=1}^k V(P_i)\Big)\le\tur{r-1}{n-p}\,.
    \end{equation}
    Putting these estimates together we obtain
    \begin{equation}
      \label{eq:TuranExt:g}
      \begin{split}
        e(G) \le
        \sum_{i=1}^k\binom{p_i}{2}
        &+
        \sum_{1\le i<j\le k} (p_i-1)p_j+(p-k)(n-m-p)\\  
        &+mk(r-2)+ \tur{r-1}{n-p}\,.
      \end{split}
    \end{equation}
    Observe that the right hand side of~\eqref{eq:TuranExt:g} defines a function, which we denote $g_n(p_1,\ldots,p_k)$, whose domain is the set of tuples (of any length $k$) of nonnegative integers. In particular we allow $k=0$,
    when~\eqref{eq:TuranExt:g} gives $g_n()=\tur{r-1}{n}$.
    
    We now give two equalities relating values of $g_n$. As a preparatory step, observe that for
    any $n'$ we have
    \begin{align}\label{eq:TId1}
      \tur{r-1}{n'+1}-\tur{r-1}{n'}
        &=n'-\Big\lfloor\frac{n'}{r-1}\Big\rfloor\;,
      \quad\text{and} \\
      \label{eq:TId2}
      \tur{r-1}{n'+r}-\tur{r-1}{n'}&=(r-1)n'+\binom{r}{2} -\Big\lfloor\frac
        {n'+r-1}{r-1}\Big\rfloor\;.
    \end{align}
    Now suppose that $k\ge1$.
    If $p_k>r$ then plugging~\eqref{eq:TId1} (with $n'=n-p=n-\sum_{\ell=1}^k p_\ell$) into
    the definition of~$g_n$ in~\eqref{eq:TuranExt:g} we obtain
    \begin{equation}\label{eq:S1}
      g_n(p_1,\ldots,p_{k-1},p_k-1)-g_n(p_1,\ldots,p_{k-1},p_k)= m-\Big\lfloor
      \frac{n-\sum_{\ell=1}^k p_\ell}{r-1}\Big\rfloor-1\,.
    \end{equation}
    Similarly, if $p_k=r$ then~\eqref{eq:TId2}
    implies
    \begin{equation}\label{eq:S2}
      g_n(p_1,\ldots,p_{k-1})-g_n(p_1,\ldots,p_{k-1},p_k)=m-\Big\lfloor
      \frac{n-\sum_{\ell=1}^k p_\ell}{r-1}\Big\rfloor-1\,.
    \end{equation}
    We note that our condition $n\le(r-1)m$ implies that $m-\big\lfloor
    \frac{n-p}{r-1}\big\rfloor-1>0$.
    
    Applying repeatedly both~\eqref{eq:S1} and~\eqref{eq:S2} we obtain
    \begin{equation*}
      g_n()-g_n(p_1,\ldots,p_k)=\sum_{i=1}^k \sum_{j=0}^{p_i-r}
      \Big(m-\Big\lfloor\frac{n-\big(p-j-\sum_{\ell=i+1}^{k}p_\ell\big)}{r-1}\Big\rfloor-1\Big)\,,
    \end{equation*}
    which together with $e(G)\le g_n(p_1,\ldots,p_k)$ and $g_n()=\tur{r-1}{n}$
    yields the desired bound on $e(G)$.
  \end{proof}
  
We are now ready to prove Case~I.

  \begin{proof}[Proofs of Theorems~\ref{thm:TuranExt} and~\ref{thm:TuranExt_stab}, Case~I]
    Let $G$ be an $n$-vertex graph and $M$ a subset of $V(G)$ of size $m$,
    where $n\le (r-1)m$, such that no~$K_r$ of~$G$ intersects~$M$. We 
 iteratively find vertex disjoint cliques $P_1,\dots,P_k$ of sizes $p_1,\ldots,p_k$
    with at least $r$ vertices as follows. Suppose that for some~$i$, the cliques $P_1,\dots,P_{i-1}$
    have already been defined. Let $P_i$ be an arbitrary maximum
    clique on at least~$r$ vertices in the graph $G-\bigcup_{j<i}P_j$. We set $k:=i-1$ and terminate if no such clique exists.
    Let $p:=\sum_{\ell=1}^k p_\ell$. Now $G$, $M$ and $P_1,\ldots,P_k$
    satisfy the conditions of Lemma~\ref{lem:largem}, so we have
    \begin{equation}\label{eq:TuranExt:eG}
      e(G)\le \tur{r-1}{n}-\sum_{i=1}^k
      \sum_{j=0}^{p_i-r} \Big(m-\Big\lfloor\frac{n-\big(p-j-\sum_{\ell=i+1}^{k}p_\ell\big)}{r-1}\Big\rfloor-1\Big)\,.
    \end{equation}
    
    We first prove Theorem~\ref{thm:TuranExt}. We distinguish two cases. First,
    $G$ contains no copy of $K_r$. In this case Tur\'an's theorem
    guarantees that $e(G)\le\tur{r-1}{n}$ with equality if and only if
    $G=\Tur{r-1}{n}$.
    
    Second, $G$ contains at least one copy of $K_r$. In this case, there is at
    least one term in the double sum in~\eqref{eq:TuranExt:eG} (since $P_1$ exists) and
    the smallest of the summands is that with $i=1$ and $j=p_1-r$, i.e., 
    \begin{equation*}
      m-\bigg\lfloor\frac{n-\big(p-(p_1-r)-\sum_{\ell=2}^{k}p_\ell\big)}{r-1}\bigg\rfloor-1
      =m-\big\lfloor\tfrac{n-r}{r-1}\big\rfloor-1
      =m-\big\lfloor\tfrac{n-1}{r-1}\big\rfloor\,.
    \end{equation*}
    Since $n\le(r-1)m$, we have $m\ge\big\lceil\tfrac{n}{r-1}\big\rceil$ and
    hence the smallest summand is at least $1$. It follows that
    $e(G)<\tur{r-1}{n}$ and so~$G$ is not extremal. This
    proves~\ref{thm:TuranExt:bound} and~\ref{thm:TuranExt:unique}.
    
    It remains to prove Theorem~\ref{thm:TuranExt_stab}. Given $\eps>0$,
    we let $\gamma^*$ be the constant given by
    Theorem~\ref{thm:BasicStability} for the input $\eps^*:=\eps/2$. We let
    \begin{equation}\label{eq:turextstab:setgammas}
      \gamma_1:=\min\big(\gamma^*,\tfrac{1}{4},\eps\big)\quad
      \text{ and }\quad \gamma:=\tfrac{\gamma_1^2}{64r^2}\,.
    \end{equation}
    
    Suppose that $e(G)\ge\turm{r-1}{n}{m}-\gamma n^2$. We may assume that
    $\gamma n^2\ge 1$, as otherwise our uniqueness result gives
    $G=\Tur{r-1}{n}$.  It follows in particular
    by~\eqref{eq:turextstab:setgammas} that $\gamma_1 n\ge 8r$, which in turn gives
    \begin{equation}\label{eq:HN}
      p-2r \ge p-\gamma_1 n/4\;.
    \end{equation}
    Observe
    that the $p-(k-1)r$ values $j+\sum_{\ell=i+1}^{k}p_\ell$
    in~\eqref{eq:TuranExt:eG} form a sequence of distinct integers, with,
    if ordered, consecutive values separated by either $1$ or $r$, and the
    smallest is~$0$. Thus at least $p/(2r)$ of these values satisfy
    $j+\sum_{\ell=i+1}^{k}p_\ell\le p/2$, or equivalently, 
    $p-j-\sum_{\ell=i+1}^{k}p_\ell\ge p/2$. In addition, as before all
    summands of the double sum in~\eqref{eq:TuranExt:eG} are non-negative. It
    follows that
    \begin{align*}
      e(G)&\le
      \tur{r-1}{n}-\frac{p}{2r}\Big(m-\Big\lfloor\frac{n-p/2}{r-1}\Big\rfloor-1\Big)
      \\&\le
      \tur{r-1}{n}-\frac{p}{2r}\Big(m-\frac{n}{r-1}+\frac{p-2r}{2r-2}\Big)
      \le\tur{r-1}{n}-\frac{p(p-2r)}{4r^2}\,,
    \end{align*}
    where we used $n\le(r-1)m$ in the last inequality.  Since
    $e(G)\ge\tur{r-1}{n}-\gamma n^2$, we can
    use~\eqref{eq:turextstab:setgammas} and~\eqref{eq:HN} to conclude $p\le\gamma_1 n/2$.
    
    Let $G'$ be the subgraph of $G$ induced
    by $V(G)\setminus\cup_{i=1}^k P_i$. We have
    \[e(G')>\tur{r-1}{n}-\gamma
    n^2-\tfrac12\gamma_1 n^2\geByRef{eq:turextstab:setgammas}\tur{r-1}{n}-\tfrac34\gamma_1
    n^2\,,\] and since
    $v(G')\ge(1-\gamma_1/2)n\geByRef{eq:turextstab:setgammas}\frac{7}{8}n$, we have
    $e(G')>\tur{r-1}{v(G')}-\gamma_1 v(G')^2$. By definition of the sets
    $P_i$ the graph $G'$
    is $K_r$-free. Therefore, by Theorem~\ref{thm:BasicStability} the graph $G'$ is $\eps^*$-close to
    $\Tur{r-1}{v(G')}$. It follows that $G$ is $\big(\eps^*
    v(G')^2+\gamma_1n^2/2\big)$-near to $\Tur{r-1}{n}$, and thus
    by~\eqref{eq:turextstab:setgammas} that $G$ is $\eps$-close to $\Tur{r-1}{n}$ as required.
  \end{proof}

\subsection{Case~II}

We first state three lemmas which we will use to prove
Theorems~\ref{thm:TuranExt} and~\ref{thm:TuranExt_stab} in Case~II. Note that the first two of these lemmas
do not require the condition $n>(r-1)m$. The first lemma asserts that every
graph~$G$ with no~$K_r$ intersecting~$M$ can easily be modified such that each
vertex outside~$M$ has high degree.

\begin{lemma}\label{lem:push}
  Let $G$ be an $n$-vertex graph and $M\subset V(G)$ have size $m$. Assume that
  no copy of $K_r$ in $G$ intersects $M$. Given $\mu\in[0,1)$, there is a graph
  $G'$ on $V(G)$ with the following properties.
  \begin{enumerate}[label=\abc]
    \item\label{lem:push:c1} $G'$ has no copy of $K_r$ intersecting $M$.
    \item\label{lem:push:c2} $e(G')\ge e(G)$, with equality if and only if
      $G=G'$.
    \item\label{lem:push:c3} Either $e(G')>e(G)+\mu^2 n^2$, or $G'$ is $\mu
      n^2$-near to $G$ (without relabelling vertices).
    \item\label{lem:push:c4} Every vertex $v\in V(G)\setminus M$ has
    $\deg_{G'}(v)\ge n-m-\mu n-1$.
  \end{enumerate} 
\end{lemma}
\begin{proof}
  We obtain $G'$ from $G$ by repeating the following procedure until
  conclusion~\ref{lem:push:c4} is satisfied. If there exists a vertex $v\in
  V(G)\setminus M$ with degree smaller than $n-m-\mu n-1$, delete all edges
  containing $v$ and insert all edges from $v$ to $V(G)\setminus
  \big(M\cup\{v\}\big)$.
  
  Observe that at each step, we add at least $\mu n$ edges to the graph,
  and edit at most $n$ edges. It follows that the algorithm terminates, and
  thus conclusions~\ref{lem:push:c2} and~\ref{lem:push:c4} get satisfied.
  Clearly, the resulting graph~$G$ also satisfies~\ref{lem:push:c1}.
  Furthermore, if the procedure is repeated more than $\mu n$
  times, then $e(G')-e(G)>\mu^2 n^2$, while otherwise the number of edits is at
  most $\mu n^2$, so conclusion~\ref{lem:push:c3} is satisfied.
\end{proof}

The next lemma states that there are few vertices which have big degree
in~$G$ and many neighbours in~$M$.

\begin{lemma}\label{lem:boundX}
  Let $G$ be an $n$-vertex graph and $M\subset V(G)$ have size~$m$. Assume that
  no copy of $K_r$ in $G$ intersects $M$. Given $\nu\in[0,1)$, let $X$ be the set
  of vertices in $G$ outside $M$ with at least $\max(1,\nu n)$ neighbours in $M$.
  Suppose that every vertex of $X$ has degree at least $n-m-\nu^2
  n$. Then we have $|X|\le(1+\nu)(r-2)m$.
\end{lemma}
\begin{proof}
  Let $x_1,\ldots,x_k$ be the vertices of a maximum clique in $G[X]$. For each
  $i\in[k]$, let $s_i$ be the number of non-neighbours of $x_i$ in $X$
  (including $x_i$ itself). Because $x_1,\ldots,x_k$ is a maximum clique, every
  vertex of $X$ is a non-neighbour of at least one $x_i$, and therefore we have
  $s_1+\ldots+s_k\ge|X|$.
  
  Observe that $x_i$ has at most $n-m-s_i$ neighbours outside
  $M$. Hence, by definition of~$X$ and since $\deg(x_i)\ge n-m-\nu^2 n$
  the vertex~$x_i$ has at least $\max\big(\nu n,s_i-\nu^2 n\big)$
  neighbours in $M$. On the other hand, no vertex of $M$ is adjacent to
  more than $r-2$ of the vertices $x_1,\ldots,x_k$, or there would be a
  copy of $K_r$ intersecting $M$. It follows that $(r-2)|M|\ge k\nu
  n$ and
  \[(r-2)|M|\ge \sum_{i=1}^k\big(s_i-\nu^2 n\big)\ge |X|-k\nu^2
  n\ge|X|-\nu(r-2)|M|\,,\]
  from which we have $|X|\le(1+\nu)(r-2)|M|$.
\end{proof}

The final preparatory lemma asserts that $\Turm{r-1}{n}{m}$ is closed under
certain local modifications.
\begin{lemma}\label{lem:Seq}
  Suppose that $n>(r-1)m$. Let $G_1\in\Turm{r-1}{n}{m}$ be a graph in which no
  $K_r$ intersects the $m$-set $M\subset V(G_1)$, and let $v\in V(G_1)\setminus
  M$ be a vertex whose neighbourhood in $G_1$ is
  $V(G_1)\setminus\big(M\cup\{v\}\big)$. Delete all edges incident to $v$ and
  insert $n-m-1$ edges, of which at least one goes to $M$. If there
  is no copy of $K_r$ intersecting~$M$ in the modified graph $G_2$, then
  $G_2\in\Turm{r-1}{n}{m}$.
\end{lemma}
\begin{proof}
  Recall that since $G_1$ is in $\Turm{r-1}{n}{m}$, it contains a copy of the graph
  $\Tur{r-1}{(r-1)m}$ with colour classes $V_1,\ldots,V_{r-1}$ which covers
  $M$, but which does not cover~$v$ because each of its vertices is either in or adjacent
  to~$M$ in~$G_1$. The same sets $V_1,\ldots,V_{r-1}$ continue to induce a
  copy of $\Tur{r-1}{(r-1)m}$ in $G_2$. Since $v$ has at least one
  $G_2$-neighbour in $M$, we can let $w_i$ be a neighbour of $v$ in $M\cap
  V_i$ for some $i$. If $v$ is adjacent to at least one vertex of each set
  $V_1,\ldots,V_{r-1}$, then letting~$w_j$ be a neighbour of $v$ in $V_j$
  for each $j\neq i$, we obtain a copy of $K_r$ in~$G_2$ intersecting $M$,
  which is a contradiction. Thus there is $j$ such that $v$ has no
  neighbours in $V_j$, and since $v$ has degree $n-m-1$ it follows that the
  neighbourhood of $v$ is precisely $V(G_1)\setminus\big(V_j\cup\{v\}\big)$. In
  other words, $v$ has the same neighbourhood as a sporadic vertex in our
  construction, and we need only to show that there is no second vertex
  $v'\neq v$ with neighbourhood $V(G_1)\setminus\big(V_j\cup\{v'\}\big)$. If
  such a vertex existed, then $v,v'$ and $w_i$ together with one
  vertex in each set $V_\ell$ with $\ell\not\in\{i,j\}$ would form a copy of $K_r$
  intersecting $M$ in $G_2$.
\end{proof}

We can now prove Case~II.

\begin{proof}[Proof of Theorems~\ref{thm:TuranExt} and~\ref{thm:TuranExt_stab}, Case~II]
  Let $G=(V,E)$ and~$M$ satisfy the conditions of the theorems.  First we show
  that $e(G)\le\turm{r-1}{n}{m}$, with equality only for graphs in
  $\Turm{r-1}{n}{m}$, which will prove Theorem~\ref{thm:TuranExt}.
  
  We apply Lemma~\ref{lem:push} to $G$ with $\mu:=0$ to obtain a graph $G'$ on
  $V$ which also has no $K_r$ intersecting $M$, which has $e(G')\ge e(G)$
  with equality only if $G=G'$, and which is such that every vertex $v\in
  V\setminus M$ has $\deg_{G'}(v)\ge n-m-1$. We now apply repeatedly the
  following further transformation to $G'$ to obtain $G''$. If there exists a
  vertex $v$ in $V\setminus M$ whose degree is $n-m-1$ and which has a
  neighbour in $M$, we delete all edges incident to $v$, and insert all edges
  from $v$ to $V\setminus\big(M\cup\{v\}\big)$. Observe that $e(G'')=e(G')$,
  and $G''$ satisfies the conditions of Lemma~\ref{lem:boundX} with $\nu:=0$. It
  follows that the set $X$ of $G''$-neighbours of $M$ in $V\setminus M$ has
  size $|X|\le(r-2)m$. Let $X'$ be a subset of $V\setminus M$ containing $X$
  of size exactly $(r-2)m$.
  
  Since $|X'\cup M|=(r-1)m$, we can now apply
  Theorem~\ref{thm:TuranExt} in Case~I to
  conclude that
  \begin{equation*}
    e\big(G''[X'\cup M]\big)
    \le\turm{r-1}{(r-1)m}{m}
    =\tur{r-1}{(r-1)m}
  \end{equation*}
  with equality only if $G''[X'\cup M]=\Tur{r-1}{(r-1)m}$. Observe that the vertices
  in $V\setminus (X'\cup M)$ are all of degree $n-m-1$ and have no neighbours in
  $M$. It follows that $e(G'')\le\turm{r-1}{n}{m}$, with equality only if
  $G''\in\Turm{r-1}{n}{m}$. Since $e(G)\le e(G')=e(G'')$, we have
  $e(G)\le\turm{r-1}{m}{n}$, with equality only if $G=G'$ and
  $G''\in\Turm{r-1}{n}{m}$. It remains only to show that if
  $e(G)=\turm{r-1}{n}{m}$, then the transformation from $G=G'$ to $G''$
  cannot take a graph outside $\Turm{r-1}{n}{m}$ to a graph in
  $\Turm{r-1}{n}{m}$. Observe that the reverse of this transformation consists
  exactly of steps satisfying Lemma~\ref{lem:Seq}, which therefore asserts that
  since $G''\in\Turm{r-1}{n}{m}$, so $G=G'\in\Turm{r-1}{n}{m}$. This proves
  Theorem~\ref{thm:TuranExt}.
  
  Finally, we prove stability, that is, Theorem~\ref{thm:TuranExt_stab}. Given $\eps>0$, set $\eps':=\eps/2$.
  Let $\gamma'$ be the constant returned by Case I of Theorem~\ref{thm:TuranExt_stab}
  for input $\eps'$ and define
  \begin{equation}\label{eq:stab2:set}
    \nu:=\gamma'\eps'^2/4\,,\qquad 
    \mu:=\nu^2/2 \qquad\text{and}\qquad
    \gamma:=\mu^2/2\,.
  \end{equation}
  Suppose that $e(G)\ge\turm{r-1}{n}{m}-\gamma n^2$. If $\gamma n^2<1$, then
  $e(G)=\turm{r-1}{n}{m}$ and so $G\in\Turm{r-1}{n}{m}$ (and in particular $G$ is
  $\eps$-close to some graph in $\Turm{r-1}{n}{m}$). It follows that we may assume
  $n\ge\gamma^{-1/2}$, and so by~\eqref{eq:stab2:set} that $\mu n\ge 1$.
  
  We apply Lemma~\ref{lem:push} to $G$ to obtain a graph $G'$ in which no copy
  of $K_r$ intersects $M$, with $e(G')\ge e(G)$, and in which every vertex $v\in
  V\setminus M$ has $\deg_{G'}(v)\ge n-m-\mu n-1$. In particular, we have
  $e(G')\le\turm{r-1}{n}{m}$. Since $\mu^2>\gamma$ by~\eqref{eq:stab2:set}, we
  must have $e(G')\le e(G)+\mu^2n^2$, so by conclusion~\ref{lem:push:c3} of
  Lemma~\ref{lem:push} the graph $G'$ is obtained from $G$ by editing at most
  $\mu n^2$ edges.
  
  Now since $\mu n\ge 1$ and by~\eqref{eq:stab2:set}, we have $\deg_{G'}(v)\ge n-m-2\mu
  n=n-m-\nu^2 n$ for each $v\in V\setminus M$. Letting
  $X$ be the vertices in $V\setminus M$ with at least $\nu n$ neighbours in
  $M$, we obtain by Lemma~\ref{lem:boundX} that $|X|\le(1+\nu)(r-2)m$. 

  Let $X'$ be a subset of $V\setminus M$ of size $(r-2)m$ which is either
  contained in $X$ (if $|X|>(r-2)m$) or contains $X$ (if
  $|X|\le(r-2)m$). We obtain a graph $G''$ by deleting all edges from
  $V\setminus (M\cup X')$ to $M$.  Observe that, since $(r-2)m<n$, the
  graph~$G''$ is obtained from $G'$ by deleting at most $(n-m-|X|)\nu
  n+\nu(r-2)m^2\le 2\nu n^2$ edges, and therefore has $e(G'')\ge e(G')-2\nu
  n^2\ge\turm{r-1}{n}{m}-\gamma n^2-2\nu n^2$ edges. Furthermore, no copy of
  $K_r$ in $G''$ intersects $M$.
  
  Let $H=G''[X'\cup M]$.
  Since there are no edges in $G''$ between $V\setminus(X'\cup M)$ and $M$,
  we have 
  \begin{equation*}\begin{split}
    e(G'')
    &= e\big(G''[V\setminus (X'\cup M)]\big) + e\big(G''\big[V\setminus(X'\cup
    M),X'\cup M\big]\big) + e(H) \\ &\le
    \binom{n-(r-1)m}{2}+\big(n-(r-1)m\big)(r-2)m+e(H)\,.
  \end{split}\end{equation*}
 Thus $e(H)\ge
  \tur{r-1}{(r-1)m}-\gamma n^2-2\nu n^2$. Furthermore,
  by Case I of Theorem~\ref{thm:TuranExt}\ref{thm:TuranExt:bound},
  $e(H)\le\tur{r-1}{(r-1)m}$.
  
  We distinguish two cases. First, $(r-1)m\ge \eps'n$. In this case, we have
  \[
  e(H)
  \ge\tur{r-1}{(r-1)m}-\frac{\gamma+2\nu}{\eps'^{2}}(r-1)^2m^2
  \geByRef{eq:stab2:set}\tur{r-1}{(r-1)m}-\gamma'(r-1)^2m^2\,.
  \]
  We apply Case~I of Theorem~\ref{thm:TuranExt_stab} to $H$
  with $\gamma'$ and $\eps'$, to obtain that $H$ is $\eps'$-close to
  $\Tur{r-1}{(r-1)m}$. 
  Second, $(r-1)m<\eps'n$. In this case, we have
  $\binom{(r-1)m}{2}<\eps'n^2$. 

  We can thus, in either case, edit at most
  $\eps'n^2$ edges of $G''$ to obtain a graph $G'''$ in which $G'''[X'\cup
  M]$ is a copy of $\Tur{r-1}{(r-1)m}$.
  Clearly, $G'''$ is a subgraph of a graph in
  $\Turm{r-1}{n}{m}$ (without sporadic vertices), and $e(G''')\ge e(G'')\ge\turm{r-1}{n}{m}-\gamma n^2-2\nu
  n^2$. It follows that we can add at most $\gamma n^2+2\nu n^2$ edges to
  $G'''$ to obtain a graph~$T$ in $\Turm{r-1}{n}{m}$. In total, we have
  made \[\mu n^2+2\nu n^2+\eps'n^2+\gamma n^2+2\nu
  n^2\leByRef{eq:stab2:set}\eps n^2\] edits from~$G$ to~$T$, and have
  preserved the property that no copy of $K_r$ intersects $M$.
\end{proof}

\section{Concluding remarks}
In our main results, we consider forbidden copies of $K_r$ that
intersect~$M$. An obvious extension would be to forbid copies of $K_r$ that intersect~$M$ in at least~$s$ vertices. We suspect that, at least for small~$s$, similar methods to those used here might give corresponding results also for this setting.
  
Another possible direction of extending Theorem~\ref{thm:TuranExt} is to forbid a general fixed $r$-partite graph~$H$, instead of~$K_r$, to touch the set $M$. The standard regularity method allows one to deduce that the upper bound from Theorem~\ref{thm:TuranExt} holds even in this case, up to an additive $o(n^2)$ term. The Tur\'an graph provides an almost matching lower bound in Case~I. The regularity method proves the corresponding counterpart to Theorem~\ref{thm:TuranExt_stab} in Case~I as well. In Case~II, however, the graphs in $\Turm{r-1}{n}{m}$ do not necessarily provide a lower bound. For example, each of the graphs in $\Turm{2}{n}{m}$ contains a copy of $C_5$ touching the set $M$. It would be interesting to determine the true extremal results in such cases.

Finally, one could ask for a stronger stability result in the spirit of~\cite{NikiPlacate1}. That is, we want to prove that if $e(G)>\tur{r-1}{n}-o(n^2)$ then after deleting $o(n)$ we get a subgraph of a graph from $\Turm{r-1}{n}{m}$. This can be obtained easily from Theorem~\ref{thm:TuranExt_stab} as follows. We take the graph $G'$ on the vertex set $V(G)$ in $\Turm{r-1}{n}{m}$ with edit distance less than $\eps n^2$ to $G$ guaranteed by Theorem~\ref{thm:TuranExt_stab}. We now remove from $V(G)$ all vertices whose neighbourhoods in $G$ and $G'$ do not have symmetric difference less than $2\sqrt{\eps}n$. Because $G$ and $G'$ are close in edit distance we remove at most $\sqrt{\eps}n$ vertices. We further remove vertices $V(G)$ that are either sporadic vertices of $G'$, or that lie in a set $V_i\cap Y$ or $V_i\setminus Y$ of size less than $4r\sqrt{\eps}n$ to obtain the vertex set $V'$, with $|V'|\ge \big(1-10r^2\sqrt{\eps}\big)n$. It is now easy to check that if $G[V']$ is not a subgraph of $G'[V']$ then there is a copy of $K_r$ in $G[V']$ touching $M$, a contradiction.

\section{Acknowledgements}
This paper was finalised during the Midsummer Combinatorial Workshop 2014 in Prague. We would like to thank the organisers for their hospitality.


\bibliographystyle{amsplain_yk} 
\bibliography{bibl}

\providecommand{\bysame}{\leavevmode\hbox to3em{\hrulefill}\thinspace}
\providecommand{\MR}{\relax\ifhmode\unskip\space\fi MR }
\providecommand{\MRhref}[2]{%
  \href{http://www.ams.org/mathscinet-getitem?mr=#1}{#2}
}
\providecommand{\href}[2]{#2}
\def\MR#1{\relax}
\begin{thebibliography}{10}

\bibitem{AllBottHlaPig}
P.~Allen, J.~B{\"o}ttcher, J.~Hladk{\'y}, and D.~Piguet, \emph{A density
  {C}orr\'adi-{H}ajnal theorem}, ar{X}iv:1403.3837.

\bibitem{AllBoettHlaPig_Turannical}
P.~Allen, J.~B\"ottcher, J.~Hladk\'y, and D.~Piguet, \emph{Tur\'annical
  hypergraphs}, Random Structures Algorithms \textbf{42} (2013), no.~1, 29--58.

\bibitem{BalBolSim}
J.~Balogh, B.~Bollob{\'a}s, and M.~Simonovits, \emph{The number of graphs
  without forbidden subgraphs}, J. Combin. Theory Ser. B \textbf{91} (2004),
  no.~1, 1--24.

\bibitem{E68}
P.~Erd\H{o}s, \emph{On some new inequalities concerning extremal properties of
  graphs}, Theory of Graphs (Proc. Colloq., Tihany, 1966), Academic Press, New
  York, 1968, pp.~77--81.

\bibitem{KS96}
J.~Koml{\'o}s and M.~Simonovits, \emph{Szemer\'edi's regularity lemma and its
  applications in graph theory}, Combinatorics, Paul Erd\H os is eighty, Vol.\
  2 (Keszthely, 1993), Bolyai Soc. Math. Stud., vol.~2, J\'anos Bolyai Math.
  Soc., Budapest, 1996, pp.~295--352.

\bibitem{Martin}
R.~Martin, \emph{The edit distance function and symmetrization}, Electron. J.
  Combin. \textbf{20} (2013), no.~3, Paper 26.

\bibitem{NikiPlacate1}
V.~Nikiforov and C.~Rousseau, \emph{Large generalized books are $p$-good,}, J.
  Combin. Theory Ser. B \textbf{92} (2004), 85--97.

\bibitem{NikiPlacate2}
V.~Nikiforov, \emph{Stability for large forbidden subgraphs}, J. Graph Theory
  \textbf{62} (2009), no.~4, 362--368.

\bibitem{S68}
M.~Simonovits, \emph{A method for solving extremal problems in graph theory,
  stability problems}, Theory of Graphs (Proc. Colloq., Tihany, 1966), Academic
  Press, New York, 1968, pp.~279--319.

\bibitem{Tur}
P.~Tur{\'a}n, \emph{Eine {E}xtremalaufgabe aus der {G}raphentheorie}, Mat. Fiz.
  Lapok \textbf{48} (1941), 436--452.

\end{thebibliography}

\end{document}